\documentclass{amsart}
\usepackage{amsmath, amssymb, amsthm}
\usepackage{url}
\usepackage{tikz}
\usepackage[hidelinks]{hyperref}
\usepackage{mathtools}
\usepackage{pgfplots}
\usepackage{mathrsfs}
\pgfplotsset{compat=1.15}
\usetikzlibrary{arrows}

\numberwithin{equation}{section}
\theoremstyle{plain}
\newtheorem{theorem}{Theorem}[section]
\newtheorem{lemma}[theorem]{Lemma}
\newtheorem{remark}[theorem]{Remark}

\newtheorem{corollary}[theorem]{Corollary}
\newtheorem{proposition}[theorem]{Proposition}
\newtheorem{example}[theorem]{Example}
\newtheorem{definition}[theorem]{Definition}

\begin{document}
	\title[Modified Distance Ratio Metrics via Domain Diameter]{Modified Distance Ratio Metrics via Domain Diameter and their geometric implications}
	    \date{\today}
	
	\author[B. Maji]{Bibekananda Maji}
	\address{Bibekananda Maji\\ Department of Mathematics \\
		Indian Institute of Technology Indore \\
		Simrol,  Indore,  Madhya Pradesh 453552, India.} 
	\email{bmaji@iiti.ac.in}
	
	\author[P. Naskar]{Pritam Naskar}
	\address{Pritam Naskar\\ Department of Mathematics \\
		Indian Institute of Technology Indore \\
		Simrol,  Indore,  Madhya Pradesh 453552, India.} 
	\email{naskar.pritam2000@gmail.com,   phd2201241008@iiti.ac.in }
	
	\author[S. K. Sahoo]{Swadesh Kumar Sahoo}
	\address{Swadesh Kumar Sahoo\\ Department of Mathematics \\
		Indian Institute of Technology Indore \\
		Simrol,  Indore,  Madhya Pradesh 453552, India.} 
	\email{swadesh.sahoo@iiti.ac.in}
	
\begin{abstract}
    Let $D\subsetneq\mathbb{R}^n,~n\ge 2$, be a domain. In this manuscript, a new version of the Vuorinen's distance ratio metric $j_D$ [{\tt J. Analyse Math.} {\bf 45} (1985), 69--115], denoted by $\zeta_D$, and a version of Gehring-Osgood's distance ratio metric $j_D'$ [{\tt J. Analyse Math.} {\bf 36} (1979), 50--74], denoted by $\zeta_D'$, are introduced to better understand how quasihyperbolic geometry interacts with bounded uniform domains in $\mathbb{R}^n$. We show that the metric $m_D$, introduced in [{\tt arXiv:2505.10964v2}], is the inner metric of $\zeta_D$ and explore their relations to several well-known hyperbolic-type metrics. The paper includes ball inclusion properties of these metrics associated with the metric $m_D$ and other hyperbolic-type metrics. The distortion properties of them are also considered under several important classes of mappings. Furthermore, as an application, we demonstrate that uniform domains can be characterized in terms of metrics $\zeta_D$ and $m_D$. 
\end{abstract}
    \maketitle
\section{\bf Introduction}
Throughout the manuscript, we consider $D\subsetneq\mathbb{R}^n,~n\ge 2$, as a domain.
In \cite{MaNaSa}, the authors studied a metric $m_D$ that agrees with the hyperbolic metric on balls and half-spaces. This metric can also be regarded as a variant of the quasihyperbolic metric $k_D$ in bounded domains, since the two coincide in every unbounded domain. 
We also investigated several geometric properties of $m_D$. These include its equivalence with the hyperbolic and quasihyperbolic metrics, the existence of geodesics, and its curvature properties. 
The distortion properties of $m_D$ under certain special mappings are also considered.
We further established a lower bound for the $m_D$-length of non-trivial closed curves in multiply connected domains, and provided characterizations of uniform domains and John disks.
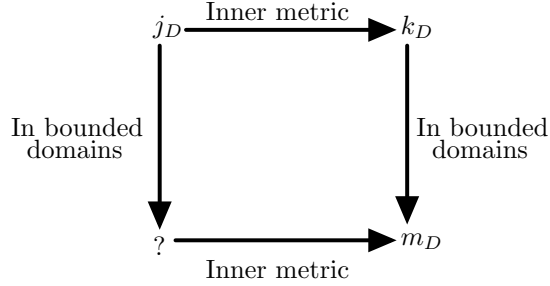
\begin{figure}[t]
\begin{tikzpicture}[line cap=round,line join=round,>=triangle 45,x=0.7cm,y=0.7cm]
	\clip(-7,-5) rectangle (6,5);
	\draw [->,line width=1.5pt] (-2.0,2) -- (2,2);
	\draw [->,line width=1.5pt] (-2.2,-2) -- (2,-2);
	\draw (-2.8,2.44) node[anchor=north west] {$j_D$};
	\draw (1.9,2.44) node[anchor=north west] {$k_D$};
	\draw (-2.8,-1.75) node[anchor=north west] {$?$};
 	\draw (1.9,-1.7) node[anchor=north west] {$m_D$};
	\draw (-1.8,2.7) node[anchor=north west] {$\textrm{Inner metric}$};
	\draw [->,line width=1.5pt] (-2.5,1.7) -- (-2.5,-1.8);
	\draw [->,line width=1.5pt] (2.2,1.7) -- (2.2,-1.7);
	\draw (2.2,0.5) node[anchor=north west] {$\textrm{In bounded}$};
	\draw (2.5,0.1) node[anchor=north west] {$\textrm{domains}$};	
	\draw (-5.5,0.5) node[anchor=north west] {$\textrm{In bounded}$};
	\draw (-5.2,0.1) node[anchor=north west] {$\textrm{domains}$};
	\draw (-1.8,-2.2) node[anchor=north west] {$\textrm{Inner metric}$};
\end{tikzpicture}
\vspace*{-1.5cm}
\caption{Motivation}\label{motivation_fig}
\end{figure}
One of the stunning results in characterizing uniform domains is due to Gehring and Osgood \cite{GO}. They showed that a uniform domain $D$ can be characterized in terms of the quasihyperbolic metric $k_D$ and the distance ratio metric $j_D$. Note that the metric $k_D$ is the inner metric of the distance ratio metric $j_D$.
Uniform domains were also characterized with respect to the metrics $m_D$ and $j_D$ in \cite{MaNaSa}, however, $j_D$ does not possess $m_D$ as its inner metric in bounded domains. Thus, a natural question arises: does there exist any metric whose inner metric is $m_D$? 

In this paper, we take the opportunity to address the above question (see Figure \ref{motivation_fig}) by introducing a new metric $\zeta_D$ designed for this purpose. Specifically, we aim for the metric $\zeta_D$ to satisfy the following properties:
\begin{enumerate}
\item[(i)] $m_D$ is the inner metric of $\zeta_D$;
\item[(ii)] $\zeta_D$ coincides with $j_D$ in any unbounded domain;
\item[(iii)] uniform domains can be characterized in terms of $m_D$ and $\zeta_D$.
\end{enumerate}

With these motivations in mind, we begin in Section \ref{preli} by recalling some preliminary results and fixing our notations. In Section \ref{Sec-Defn}, we present two constructions of new metrics, one of which leads to the desired metric $\zeta_D$ and the other one, denoted by $\zeta_D'$, is equivalent to $\zeta_D$ and it generalizes $j_D'$ of Gehring and Osgood. We also compute explicit expressions for this metric in the unit disk, annulus, and punctured unit disk for completeness.
Section \ref{cmprsn with other metrics} explores the relationship between these two metrics and several well-known metrics, including the generalized hyperbolic metric $m_D$, the hyperbolic metric $h_B$ on any ball $B$, the distance ratio metric $j_D$, and the quasihyperbolic metric $k_D$. 
Section \ref{inner metric-sec} establishes that $\zeta_D$ has $m_D$ as its inner metric. In Section \ref{ball inclusion}, we study ball inclusion properties involving $\zeta_D$, $\zeta_D'$, and $m_D$. 
Next, in Section \ref{behvr undr spcl maps}, we analyze the behavior of $\zeta_D$ under certain special maps. Section \ref{charc-of-domains} presents inequalities involving $\zeta_D$ that help characterize uniform domains. We conclude the paper with final remarks.


\section{\bf Notations and preliminary results}\label{preli}
Throughout this article, we use the following notations:
	\begin{align*}
	 &\mathrm{diam}(D)=\begin{cases}
			d(D), &\text{if $D$ is bounded,}\\
			\infty, &\text{otherwise.} 
		\end{cases}\\
	 &\partial D:=\text{Boundary of}\ D.\\
	 &\delta(z)=\delta_D(z):=\min\left\{|z-\xi|:\xi\in \partial D\right\},\ \text{if there is no confusion about the domain } D.\\
	 &\ell(\gamma):=\text{Euclidean length of }\gamma.\\ 
	 &B=B(z_0,r):=\left\{z\in\mathbb{R}^n:|z-z_0|<r\right\}\textrm{ and } \mathbb{B}=B(0,1).\\
	 &\mathbb{D}:=\left\{z\in\mathbb{C}:|z|<1\right\}.\\
	 &\Gamma_{xy}:=\text{the set of all rectifiable paths in the respective domain joining } x \text{ to } y.\\
	 &a\wedge b:=\min\left\{a,b\right\}.
	\end{align*}
Let us next recall the hyperbolic-type metric $m_D$, which is introduced in \cite{MaNaSa}. For any two points $x,y$ in $D$, it is defined as
	\begin{equation}\label{mDdefn}
		m_{D}(x,y):=\inf_{\gamma\in \Gamma_{xy}}\int_{\gamma}m_D(z)\,|dz|,
	\end{equation}
	where $m_D(z)=d(D)/[\delta(z)(d(D)-\delta(z))]$. As noted in \cite{MaNaSa},
	the metric $m_D$ matches with the hyperbolic metric in balls and half-spaces, whereas it agrees with the quasihyperbolic metric in any unbounded domain $D$. The hyperbolic metrics in a ball $B=B(z_0,r)$ and the upper-half space $\mathbb{H}:=\{z=(z_1,z_2,\ldots,z_n)\in\mathbb{R}^n:\,z_n>0\,\}$ are defined as follows:
	\begin{align}\label{hyperbolic metric}
		h_B(x,y):=\inf_{\gamma\in \Gamma_{xy}}\int_{\gamma}h_B(z)\,|dz|
        ~~\mbox{ and }~~
        h_{\mathbb{H}}(x,y):=\inf_{\gamma\in \Gamma_{xy}}\int_{\gamma}\frac{|dz|}{z_n},
 	\end{align} 
    respectively, with $h_B(z)={2r}/(r^2-|z-z_0|^2)$.
	    The generalization of the hyperbolic metric in the upper half-plane model is known as the quasihyperbolic metric, initiated by Gehring and Palka in \cite{GP}, in the following form:
	\begin{equation}\label{kmetric}
		k_D(x,y):=\smashoperator{\inf_{\gamma\in \Gamma_{xy}}}\hspace{0.3 cm}\int_{\gamma}\frac{|dz|}{\delta_D(z)},
	\end{equation}
	for any two points $x,y$ in $D$. 
    
    The $d$-length, $\ell_d(\gamma)$, of a curve $\gamma:[0,1]\rightarrow X$ in a metric space $(X,d)$ is defined as
	\begin{align*}
		\ell_d(\gamma)=\sup_{\mathcal{P}}\sum_{i=1}^{n}d(\gamma(t_{i-1}),\gamma(t_i)),
	\end{align*}	
	where $0=t_0\leq t_1\leq\dots\leq t_n=1$ is a member of $\mathcal{P}$, the set of all partitions of $[0,1]$. If $\ell_d(\gamma)<\infty$, then we call $\gamma$ rectifiable. The {\em inner metric,} $\tilde{d}$, of $d$ is defined as
	\begin{align*}
		\tilde{d}(x,y)=\inf_{\gamma\in \Gamma_{xy}} \ell_d(\gamma).
	\end{align*}
	The repeated use of triangle inequalities yields $d\le \tilde{d}$. If the reverse inequality is also true, i.e., $\tilde{d}=d$, then $d$ is called a {\em path metric}. It is well-known that the distance ratio metric $j_D$, defined by Vuorinen in \cite{Vu85}, is not a path metric and $\tilde{j}_D=k_D$ (see, for instance, \cite[Lemma~5.3]{Has03} and \cite[Theorem 3.7 (1) and (3)]{Vai96}), where the metric $j_D$ is defined as
	\begin{align}\label{defn-of-j}
		j_D\left(x,y\right)=\log\left(1+\frac{|x-y|}{\delta(x)\wedge \delta(y)}\right).
	\end{align}
	Another version of the distance ratio metric appears in \cite{GO} as
	\begin{align}\label{defn-of-j'}
		j'_D\left(x,y\right)=\frac{1}{2}\log\left\{\left(1+\frac{|x-y|}{\delta(x)}\right) \left(1+\frac{|x-y|}{\delta(y)}\right)\right\}.
	\end{align}
	The following are well-known facts, but we are including the results for ease of reading.
	\begin{lemma}\label{j&j' equiv & reln wth k}
		In any domain $D\subsetneq\mathbb{R}^n$, we have
		\begin{enumerate}
		\item[(i)] $j'_D(x,y)\leq j_D(x,y)\leq 2\,j'_D(x,y)$; and
		\item[(ii)] $j_D(x,y)\leq k_D(x,y)$,
		for all $x,y\in D$.
	    \end{enumerate}
	\end{lemma}
The reader is referred to the book \cite{HaKlVu} for additional details on these hyperbolic-type metrics and their different characteristics.
			
	We now recall the definition of a uniform domain introduced by Martio and Sarvas \cite{MS79}. 
    \begin{definition}
    A domain $D$ is said to be a {\em uniform domain}  if there exists a constant $c\geq 1$ with the property that any two points $x,y\in D$ can be joined by a rectifiable path $\gamma$ such that the following two hold:
	$$
		\ell(\gamma) \leq c\,|x-y| ~~\mbox{ and }~~\ell(\gamma[x,z]) \wedge\ell(\gamma[z,y])\leq c\,\delta(z) \text{ for all $z\in\gamma$}.
		$$
\end{definition}
    In 1979, Gehring and Osgood \cite[Corollary 1]{GO} proved that a domain $D$ is uniform if and only if there are constants $c,d\geq 1$ such that for any two points $x,y\in D$
	$$k_D(x,y)\leq c\,j'_D(x,y)+d.$$
	Later it has been shown by Vuorinen in \cite[Example 2.50 (2)]{Vu85} that the above characterization is equivalent to the inequality
    \begin{equation}\label{GO-uniform}
    k_D(x,y)\leq c'\,j_D(x,y),
    \end{equation}
    with $c'\geq1$ depends only on $c$ and $d$.
    
    Let $D,D'$ be domains in $\mathbb{R}^n$. The linear dilatation of a homeomorphism $f:D\rightarrow D'$ at a point $x\in D$ is defined by
    \begin{align}\label{defn of dilatation}
    	H(f,x):=\limsup_{r\rightarrow 0}\cfrac{\sup\left\{ |f(x)-f(y)|:|x-y|=r\right\}}{\inf\left\{ |f(x)-f(z)|:|x-z|=r\right\}}.
    \end{align}
    The map $f$ is said to be {\em $K$-quasiconformal}, with $1\le K<\infty$, if $\sup_{x\in D} H(f,x)<K$.    
    If the map $f$ is not a homeomorphism, then we call it a {\em quasiregular mapping}. Note that the case $K=1$ gives the conformality of $f$. It is appropriate here to remark that due to Heinonen and Koskela \cite[Theorem 1.4]{HK}, the $\limsup$ in the definition of quasiconformality can be replaced with $\liminf$, and hence with only limit.  
    For the basic theory of quasiconformal mappings, the readers are referred to \cite{AhlBook} for the planar case and \cite{VaiBook,VuoBook} for the higher dimension. 

    We conclude this section by recalling the frequently used Bernoulli inequalities, which hold for $t \geq 0$:  
\begin{align*}
\log(1+at) &\leq a \log(1+t), \text{ for } a \geq 1, \\
\log(1+at) &\geq a \log(1+t), \text{ for } 0 \leq a \leq 1.
\end{align*}
    

\section{\bf The new metric $\zeta_D$ as an analogue of $j_D$}\label{Sec-Defn}	
	In this section, our main objective is to introduce two metrics $\zeta_D$ and $\zeta'_D$ in a domain $D$ which agree with $j_D$ and $j'_D$, respectively, when $D$ is unbounded.
 \subsection{Definition and basic properties}       
	For two points $x,y\in D$, we define 
	\begin{align}\label{defn of zeta}
		&\zeta_D(x,y):=\log\left(1+\frac{d(D)\,|x-y|}{\eta_D(x)\wedge\eta_D(y)}\right),
	\end{align}
	and
    \begin{align}\label{defn of zeta'}
		\zeta'_D(x,y):&=\frac{1}{2}\log\left\{\left(1+\frac{d(D)\,|x-y|}{\eta_D(x)}\right)\left(1+\frac{d(D)\,|x-y|}{\eta_D(y)}\right)\right\},
	\end{align}
	where $\eta_D(z):=\delta_D(z)(d(D)-\delta_D(z))$.
	
	It is easy to observe that $\zeta_D$ and $\zeta'_D$ coincide with $j_D$ and $j'_D$, respectively, in the limit as $d(D)\to\infty$. The next step is to establish that both $\zeta_D$ and $\zeta'_D$ are indeed metrics. Some recent work of Mocanu \cite{Moc24,Moc25} addresses similar metric constructions arising from $1$-Lipschitz functions on $D$. However, for our purposes, we require more general versions of Mocanu’s results, which in turn yield the two new metrics $\zeta_D$ and $\zeta'_D$. In this direction, we obtain the following result:
	
	\begin{theorem}\label{Two general metrics}
		Let $D$ be a proper subdomain in $\mathbb{R}^n$, $n\geq 2$. If $f$ is a positive $\alpha$-Lipschitz function on $D$ with $\alpha>0$, then 
		\begin{enumerate}
			\item [(i)] $d(x,y)=\log\left(1+\cfrac{\alpha\,|x-y|}{f(x)\wedge f(y)}\right)$, and
			\item [(ii)] $d'(x,y)=\cfrac{1}{2}\log\left\{\left(1+\cfrac{\alpha\,|x-y|}{f(x)}\right)\left(1+\cfrac{\alpha\,|x-y|}{f(y)}\right)\right\}$
		\end{enumerate}
		define metrics on $D$.
	\end{theorem}
	\begin{proof} We only need to show the triangle inequality for both of them, as all other metric properties are trivial.
		\begin{enumerate}
	\item[(i)] We need to prove that
		$$d(x,y)\leq d(x,z)+d(z,y),$$
		which is equivalent to
		\begin{align*}
			1+\frac{\alpha\,|x-y|}{f(x)\wedge f(y)}\leq \left(1+\frac{\alpha\,|x-z|}{f(x)\wedge f(z)}\right)&\left(1+\frac{\alpha\,|z-y|}{f(z)\wedge f(y)}\right)\\
			\iff	\frac{|x-y|}{f(x)\wedge f(y)}\leq \frac{|x-z|}{f(x)\wedge f(z)}+\frac{|z-y|}{f(z)\wedge f(y)}+&\frac{\alpha\,|x-z||z-y|}{(f(x)\wedge f(z))(f(z)\wedge f(y))}.
		\end{align*}
		 Without loss of generality, we may assume that $f(x)\leq f(y)$. Therefore, it is enough to show that 
		\begin{align}\label{rqrdineqlty-1}
			\frac{|x-y|}{f(x)}\leq \frac{|x-z|}{f(x)\wedge f(z)}+\frac{|z-y|}{f(z)\wedge  f(y)}+\frac{d(D)\,|x-z||z-y|}{(f(z)\wedge f(y))(f(x)\wedge f(z))}.
		\end{align}
		
		{\bf Case-I:} Suppose that $f(z)\leq f(x)$. Then $f(x)\wedge f(z)=f(z)=f(z)\wedge f(y)$ and by the Euclidean triangle inequality, we obtain
		\begin{align*}
			|x-y|\leq|x-z|+|z-y|\Rightarrow \frac{|x-y|}{f(x)}\leq \frac{|x-y|}{f(z)}\leq \frac{|x-z|}{f(z)}+\frac{|z-y|}{f(z)},
		\end{align*}
		which gives the required inequality \eqref{rqrdineqlty-1} in this case.
		
		{\bf Case-II:} Let us assume $f(x)\leq f(z)$. Then the right-hand side of \eqref{rqrdineqlty-1} becomes
		\begin{align*}
			\frac{|x-z|}{f(x)}+\frac{|z-y|}{f(z)\wedge f(y)} &+\frac{\alpha\,|x-z||z-y|}{f(x)(f(z)\wedge f(y))}\\
			&=\frac{1}{f(x)}\left\{|x-z|+|z-y|\left(\frac{f(x)+\alpha\,|x-z|}{f(z)\wedge f(y)}\right)\right\}.
		\end{align*}
		Since the function $f$ is $\alpha$-Lipschitz, $f(z)-f(x)\leq \alpha\,|z-x|$. Hence, we have 
			$$f(x)+\alpha\,|x-z|
			\geq f(z)\geq f(z)\wedge f(y)
			\implies \, \frac{f(x)+\alpha\,|x-z|}{f(z)\wedge f(y)}\geq 1.
            $$
				Finally, we obtain \eqref{rqrdineqlty-1}, which follows from the preceding inequality together with the Euclidean triangle inequality. Indeed, we have
		\begin{align*}
			\frac{1}{f(x)}\left\{|x-z|+|z-y|\left(\frac{f(x)+\alpha\,|x-z|}{f(z)\wedge f(y)}\right)\right\}
			\geq \frac{1}{f(x)}\left\{|x-z|+|z-y|\right\}\geq\frac{|x-y|}{f(x)}.
		\end{align*}
		This completes the proof.
			
	\item[(ii)] For the function $d'$, the triangle inequality is equivalent to
		\begin{align}\label{rqrdineqlty-2}
			\left(1+\frac{\alpha|x-y|}{f(x)}\right)
			\left(1+\frac{\alpha|x-y|}{f(y)}\right)
			\leq \left(1+\frac{\alpha|x-z|}{f(x)}\right)&
			\left(1+\frac{\alpha|x-z|}{f(z)}\right)\nonumber\\
			&\left(1+\frac{\alpha|z-y|}{f(z)}\right)
			\left(1+\frac{\alpha|z-y|}{f(y)}\right).
		\end{align}
		By applying the Euclidean triangle inequality on the left-hand side of \eqref{rqrdineqlty-2}, we get
		\begin{align*}
			\left(1+\frac{\alpha|x-y|}{f(x)}\right)
			\left(1+\frac{\alpha|x-y|}{f(y)}\right)
			& \leq \left(1+\alpha\,\frac{|x-z|+|z-y|}{f(x)}\right)\left(1+\alpha\,\frac{|x-z|+|z-y|}{f(x)}\right)\\           
			&\hspace{0.5cm}+2\alpha^2\,\frac{|x-z||z-y|}{f(x)f(y)}+\alpha^2\,\frac{|x-z|^2}{f(x)f(y)}+\alpha^2\,\frac{|z-y|}{f(x)f(y)}\\            
			&=\left(1+\frac{\alpha|x-z|}{f(x)}\right)
			\left(1+\frac{\alpha|z-y|}{f(y)}\right)
			+\alpha\,\frac{|x-z|}{f(y)}\left(1+\frac{\alpha|x-z|}{f(x)}\right)\\            
			&\hspace{0.5cm}+\alpha\,\frac{|z-y|}{f(x)}\left(1+\frac{\alpha|z-y|}{f(y)}\right)+\alpha^2\,\frac{|x-z||z-y|}{f(x)f(y)}\\            
			&\leq \left(1+\frac{\alpha|x-z|}{f(x)}\right)
			\left(1+\frac{\alpha|z-y|}{f(y)}\right)
			+\alpha\,\frac{|x-z|}{f(z)}\left(1+\frac{\alpha|x-z|}{f(x)}\right)\times\\            
			&\hspace{0.6cm}\left(1+\frac{\alpha|z-y|}{f(y)}\right)
			+\alpha\,\frac{|z-y|}{f(z)}\left(1+\frac{\alpha|x-z|}{f(x)}\right)\left(1+\frac{\alpha|z-y|}{f(y)}\right)\\            
			&\hspace{0.5cm}+\alpha^2\,\frac{|x-z||z-y|}{f(z)^2}\left(1+\frac{\alpha|x-z|}{f(x)}\right)\left(1+\frac{\alpha|z-y|}{f(y)}\right),
		\end{align*}
which is exactly the right-hand side of \eqref{rqrdineqlty-2}. The last step follows from the two inequalities: 
$$\frac{\alpha|x-z|}{f(x)}\leq \frac{\alpha|x-z|}{f(z)}\left(1+\frac{\alpha}{f(x)}|x-z|\right)$$
and
$$\frac{\alpha|z-y|}{f(y)}\leq \frac{\alpha|z-y|}{f(z)}\left(1+\frac{\alpha}{f(y)}|z-y|\right),
$$
which follow from the fact that 
$f$ is $\alpha$-Lipschitz and hence 

$$
f(x)+\alpha\,|x-z|\geq f(z)\ 
	\implies 
	\frac{\alpha}{f(z)} \geq \frac{\alpha}{f(x)+\alpha\,|x-z|}=\frac{\alpha/f(x)}{1+\left(\alpha/f(x)\right)\,|x-z|}.
$$  
\end{enumerate}
Hence, $d'$ defines a metric, and this completes the proof.	
\end{proof}

\begin{remark}
	Theorem \ref{Two general metrics} is true in a general metric space setting, which need not be Euclidean space.
\end{remark}
	
We now prove a lemma containing some useful properties of the functions $\delta_D$ and $\eta_D$, which leads us to the fact that $\zeta_D$ and $\zeta'_D$ are metrics. We use $\delta(z)$ and $\eta(z)$ instead of $\delta_D(z)$ and $\eta_D(z)$, respectively, if there is no confusion about the domain from which the point $z$ belongs.
	\begin{lemma}\label{Prop-of-eta-delta}
		Let $x,y$ be any two points in a bounded domain $D$ in $\mathbb{R}^n$, $n\geq 2$. Then the following properties hold:
		\begin{enumerate}
			\item [(i)] If $\delta(x)\leq\delta(y)$, then we have $\eta(x)\leq\eta(y)$,
			\item[(ii)] The function $\delta$ is $1$-Lipschitz,
			\item[(iii)] The function $\eta$ is $d(D)$-Lipschitz, and
			\item [(iv)] $\eta(x)\leq d(D)\,\delta(x)\leq d(D)\,|x-\xi|$ for any point $x\in D$ and $\xi\in \mathbb{R}^n\setminus D$.
		\end{enumerate}		
	\end{lemma}
	\begin{proof}
		We fix any two distinct points $x$ and $y$ in $D$ such that $\delta(x)\leq \delta(y)$, without any loss of generality.
		\begin{enumerate}
		\item [(i)]
		 Let $d=d(D)$ and consider the function $f:\left(0,d/2\right]\rightarrow\left(0,d^2/4\right]$ defined by $f(t)=t(d-t)$. Since $t\leq d/2$, $f'(t)\geq 0$ and hence $f$ is increasing. Then the fact $\delta(x)\leq\delta(y)$ concludes the result.
		 
		 \item[(ii)] For any $\xi\in \partial D$, we have $\delta(x)\leq|x-\xi|\leq|x-y|+|y-\xi|$. Taking infimum over all $\xi\in\partial D$, one can obtain $\delta(x)-\delta(y)\leq|x-y|$. Since the role of $x$ and $y$ can be interchanged, we have the required result.
		 		 
		\item [(iii)] It is enough to prove the function $f$, considered in (i), is $d$-Lipschitz. To see this, we take two points $t_1, t_2\in \left(0,d/2\right]$ and calculate
		\begin{align*}
			|f(t_1)-f(t_2)|=|d\,(t_1-t_2)-(t_1^2-t_2^2)|=|t_1-t_2||d-(t_1+t_2)|\leq d\,|t_1-t_2|.
		\end{align*}
		Now, taking $t_1=\delta(x)$ and $t_2=\delta(y)$, we obtain
		\begin{align*}
			|f(\delta(x))-f(&\delta(y))|\leq d\,|\delta(x)-\delta(y)|\leq d\,|x-y|
			\ \implies\ |\eta(x)-\eta(y)|\leq d\,|x-y|,
		\end{align*}
		from (ii). Thus, the function $\eta$ is $d(D)$-Lipschitz.
		
		\item [(iv)] We have $d(D)-\delta(x)\leq d(D)\implies\eta(x)\leq d(D)\delta(x)$. Also, for any $\xi\in \mathbb{R}^n\setminus D$, it follows from the definition of $\delta(x)$ that $\delta(x)\leq |x-\xi|$. 
        \end{enumerate}
        This completes the proof.
	\end{proof}	
		
	\begin{proposition}
		For a domain $D\subsetneq \mathbb{R}^n$, $n\geq 2$, the expressions defined in \eqref{defn of zeta} and \eqref{defn of zeta'} give two metrics on $D$. 
	\end{proposition}
	\begin{proof}
		For the case $d(D)<\infty$, the proof follows from Theorem \ref{Two general metrics} and Lemma \ref{Prop-of-eta-delta} (iii) with $f(x)=\eta_D(x)$. Note that, whenever $d(D)=\infty$, $d(D)/\eta(x)$ matches with $1/\delta(x)$ and hence in this case also \eqref{defn-of-j}$\approx$\eqref{defn of zeta} and \eqref{defn-of-j'}$\approx$\eqref{defn of zeta'} define metrics, due to Lemma \ref{Prop-of-eta-delta} (ii) and Theorem \ref{Two general metrics}.
	\end{proof}
	
\begin{proposition}\label{eqvlnc of two metric}
	Let $D$ be a proper subdomain of $\mathbb{R}^n$. Then we have the following relation between the metrics $\zeta_D$ and $\zeta'_D$. For any two points $x,y\in D$, we have
	\begin{align*}
		\zeta'_D(x,y)\leq\zeta_D(x,y)\leq2\,\zeta'_D(x,y).
	\end{align*}
	Both inequalities are sharp.
\end{proposition}
\begin{proof}
	Without loss of generality, let us assume $x,y$ be any two points in $D$ with $\eta(x)\leq \eta(y)$. Then we have 
	\begin{align*}
		1+&\frac{d(D)|x-y|}{\eta(y)}\leq 1+\frac{d(D)|x-y|}{\eta(x)}\\
		\implies \left(1+\frac{d(D)|x-y|}{\eta(x)}\right)&\left(1+\frac{d(D)|x-y|}{\eta(y)}\right)\leq\left(1+\frac{d(D)|x-y|}{\eta(x)}\right)^2.
	\end{align*}
	Taking logarithm on both sides yields the first inequality. The sharpness follows from the fact that $\zeta'_{\mathbb{D}}(-t,t)=\zeta_{\mathbb{D}}(-t,t)$, for $0<t<1$.
	
	To see the second inequality, we use the fact 
	$$\left(1+\frac{d(D)|x-y|}{\eta(y)}\right)\geq 1,$$
	which implies
	\begin{align*}
		\left(1+\frac{d(D)|x-y|}{\eta(x)}\right)\left(1+\frac{d(D)|x-y|}{\eta(y)}\right)\geq\left(1+\frac{d(D)|x-y|}{\eta(x)}\right).
	\end{align*}
	Finally, taking logarithm on both sides yields the required inequality. Since $\zeta'_{\mathbb{D}}(0,t)/\zeta_{\mathbb{D}}(0,t)\rightarrow 1/2$ as $t\rightarrow 1$, the sharpness follows. 
\end{proof}

\begin{remark}
	The well-known equivalence between $j_D$ and $j_D'$ in Lemma $\ref{j&j' equiv & reln wth k}$ {\normalfont(i)} follows just by taking $d(D)\rightarrow\infty$ in the above proposition.
\end{remark}

\subsection{Examples}		
We now see the formula for the metric $\zeta_D$ in the unit disk, annulus, and the punctured unit disk.		
\begin{example}
	\normalfont
 Consider $D=\mathbb{D}$. Then for any two points $x,y\in\mathbb{D}$, we have
	\begin{align*}
		\zeta_{\mathbb{D}}(x,y)=\log\left(1+\frac{2|x-y|}{(1-|x|^2)\wedge(1-|y|^2)}\right).
	\end{align*}
	In particular, if we take two points, say, $x=t>0$ and $y=-s<0$ with $t\geq s$, then
	\begin{align*}
		\zeta_{\mathbb{D}}(t,-s)=\log\left(1+\frac{2(t+s)}{1-t^2}\right).
	\end{align*}
	Comparing the formulas of $\zeta_{\mathbb{D}}$ and $h_{\mathbb{D}}$ directly is not a trivial	task. Instead, one can use Theorem \ref{equiv of zeta&m} and the fact that $m_{\mathbb{D}}=h_{\mathbb{D}}$ (see Corollary \ref{equiv of zeta&h in unit disk}).
\end{example}	
	
\begin{example}\label{Annls exmpl}
	\normalfont	
 Let us consider the annulus $A_{r,R}:=\left\{z\in\mathbb{C}:0<r<|z|<R\right\}$. It can be checked that for any $z\in A_{r,R}$, we have
\begin{align*}
	\delta_{A_{r,R}}(z)=\begin{cases}
		|z|-r, & \text{if $r<|z|\leq \cfrac{r+R}{2}$},\\
		R-|z|, & \text{if $\cfrac{r+R}{2}\leq|z|<R$}.
	\end{cases}
\end{align*}
Therefore, by direct calculation, we have
\begin{align*}
	\zeta_{A_{r,R}}(x,y)=\log\left(1+\frac{2R|x-y|}{\eta_{A_{r,R}}(x)\wedge\eta_{A_{r,R}}(y)}\right), 
\end{align*}
with
\begin{align*}
	\eta_{A_{r,R}}(z)=\begin{cases}
		(|z|-r)(2R+r-|z|), & \text{if $r<|z|\leq \cfrac{r+R}{2}$},\\
		R^2-|z|^2, & \text{if $\cfrac{r+R}{2}\leq|z|<R$}.	\end{cases}
\end{align*}	
For a partial comparison of the $\zeta$-metric with the hyperbolic metric, one can see \cite[Corollary 3.2]{MaNaSa} and use Theorem \ref{equiv of zeta&m}.
\end{example}
\begin{example}
	\normalfont	
Consider $\mathbb{D}^*:=\mathbb{D}\setminus\left\{0\right\}$. In this case, the metric can be obtained by taking $R\rightarrow1$ and $r\rightarrow0$ in Example \ref{Annls exmpl}. Indeed, we have 
\begin{align*}
	\zeta_{\mathbb{D}^*}(x,y)=\log\left(1+\frac{2|x-y|}{\eta_{\mathbb{D}^*}(x)\wedge\eta_{\mathbb{D}^*}(y)}\right), 
\end{align*}
where
	\begin{align*}
		\eta_{\mathbb{D}^*}(z)=\begin{cases}
			2|z|-|z|^2, & \text{if $0<|z|\leq \cfrac{1}{2}$},\\
			1-|z|^2, & \text{if $\cfrac{1}{2}\leq|z|<1$}.
		\end{cases}
	\end{align*}
\end{example}
One can use Theorem \ref{equiv of zeta&m} and \cite[Example 3.3]{MaNaSa} to see the relation between $\zeta_{\mathbb{D}^*}$ and $h_{\mathbb{D}^*}$ to some extent.


\section{\bf Comparison with other hyperbolic-type metrics}\label{cmprsn with other metrics}	
	Here we study the relation of both the metric $\zeta_D$ and $\zeta'_D$ with some of the closely related hyperbolic-type metrics, viz., the $m_D$-metric, the hyperbolic metric $h_D$, the distance ratio metrics $j_D$ and $j'_D$, and the quasi-hyperbolic metric $k_D$.
    \subsection{Comparison with the $m_D$ and $h_D$-metrics} First we compare $\zeta_D$ and $\zeta_D'$ with the $m_D$-metric in any domain $D$.
	\begin{theorem}\label{equiv of zeta&m}
		For any two points $x,y$ in a domain $D\subsetneq\mathbb{R}^n$, we have
		\begin{align}
			m_D(x,y)\geq\zeta_D(x,y)\geq \zeta'_D(x,y).
		\end{align}
	\end{theorem}
	\begin{proof}
		We may assume, without loss of generality, that $\delta(x)\leq\delta(y)$. By Lemma \ref{Prop-of-eta-delta}(i), we have $\eta(x)\leq\eta(y)$. Let $\gamma:[0,1]\rightarrow D$ be a rectifiable arc with end points $x$ and $y$. Then one can obtain
		\begin{align*}
			\int_{\gamma}m_D(z)|dz|
			=&d(D)\,\int_0^1\frac{|d\gamma(t)|}{\eta(\gamma(t))}\\
			\geq &d(D)\,\int_0^1\frac{|d(\gamma(t)-x)|}{\eta(x)+d(D)|\gamma(t)-x|}\\
			=&\log\left(1+\frac{d(D)|x-y|}{\eta(x)}\right)\\
			=&\zeta_D(x,y).
		\end{align*} 
		Taking infimum over all such arcs $\gamma$, the first inequality follows. The second inequality holds due to Proposition \ref{eqvlnc of two metric}.
	\end{proof}
	
	\begin{remark}
		We show in Theorem \ref{uniform charc} that the reverse of the first inequality is also true in a uniform domain, up to some constant.
	\end{remark}

\begin{corollary}
    Let $D$ be a simply connected planar domain. Then we have $h_D(x,y)\geq \frac{1}{4}\zeta_D(x,y)\geq\frac{1}{4}\zeta_D'(x,y)$.
\end{corollary}
\begin{proof}
    The proof follows from \cite[(8.4)]{BM} and \cite[Theorem 3.5]{MaNaSa}.
\end{proof}
	\begin{corollary}
			For all $x,y\in B$, we have $h_B(x,y)\geq\zeta_B(x,y)\geq\zeta_B'(x,y)$.
	\end{corollary}
    \begin{proof}
        Since $m_D$ matches with $h_D$ on any Euclidean ball $B$, the inequalities follow.
    \end{proof}

    \begin{corollary}\label{equiv of zeta&h in unit disk}
		For any two points $x,y\in \mathbb{B}$, we have $\zeta_{\mathbb{B}}(x,y)\leq h_{\mathbb{B}}(x,y)\leq2\,\zeta_{\mathbb{B}}(x,y)$.
	\end{corollary}
	\begin{proof}
		The proof follows from the above remark, \cite[Lemma 4.9]{HaKlVu}, and Theorem \ref{equiv of zeta&j}.
	\end{proof}
Though the following fact is well-known, due to \cite[Lemma 2.1]{GP}, we derive it as a corollary of Theorem \ref{equiv of zeta&m}.  		
	\begin{corollary}\label{j-leq-k}
		For any two points $x,y$ in $D\subsetneq\mathbb{R}^n$, $n\geq 2$, we have $k_D(x,y)\geq j_D(x,y)\geq j_D'(x,y)$.
	\end{corollary}
	\begin{proof}
		The result follows by taking $d(D)\rightarrow\infty$ in Theorem \ref{equiv of zeta&m}.
	\end{proof}
    
	A short proof of the following fact can be obtained easily by Theorem \ref{equiv of zeta&m}. We proved this result initially in \cite[Corollary 4.3]{MaNaSa} using a lengthy technique.
	\begin{corollary}
		For any two point $x,y$ in $D\subsetneq\mathbb{R}^n$, we have
		\begin{align}\label{m and eta/eta inequality}
			m_D(x,y)\geq\left|\log\frac{\eta(y)}{\eta(x)}\right|.
		\end{align}
	\end{corollary}
	\begin{proof}
		If $\eta(x)\leq \eta(y)$, then from Theorem \ref{equiv of zeta&m} we have
		\begin{align*}
			m_D(x,y)\geq \log\left(1+\frac{d(D)|x-y|}{\eta(x)}\right)\geq\log\left(\frac{\eta(x)+d(D)|x-y|}{\eta(x)}\right)\geq\log\frac{\eta(y)}{\eta(x)},
		\end{align*} 
		where the last step follows from Lemma \ref{Prop-of-eta-delta} (ii). Similarly, if $\eta(y)\leq \eta(x)$, we have $m_D(x,y)\geq\log(\eta(x)/\eta(y))$. Hence the result follows.
	\end{proof}	
   
 \subsection{Comparison with the $j_D$ and $j'_D$-metrics} 
 We have seen that $j_D$ and $\zeta_D$ agree in an unbounded domain. The following theorem gives their equivalence in any bounded domain.	
	\begin{theorem}\label{equiv of zeta&j}
		For any two points $x,y$ in a bounded domain $D$, we have
		\begin{enumerate}
			\item [(i)] $j_D(x,y)\leq\zeta_D(x,y)\leq 2j_D(x,y)$, and
		    \item [(ii)] $j'_D(x,y)\leq\zeta'_D(x,y)\leq 2\,j'_D(x,y)$,
	   \end{enumerate}
	   where all the constants are the best possible.
	\end{theorem}
	\begin{proof}
		Let us fix two points $x$ and $y$ in $D$. To prove the inequalities, it suffices to consider the case $\delta(x)\leq\delta(y)$. Then from Lemma \ref{Prop-of-eta-delta} (i), we have $\eta(x)\leq\eta(y)$.
        \begin{enumerate}
	\item[(i)] 
		Upon our assumption, we have
		\begin{align*}
			j_D(x,y)=\log\left(1+\frac{|x-y|}{\delta(x)}\right).
		\end{align*}
		First note that, for any point $z\in D$, the facts $\delta(z)\leq d(D)/2$ and $d(D)-\delta(z)<d(D)$ provide us
		\begin{align}\label{mdelta}
			1<\frac{d(D)}{d(D)-\delta(z)}\leq 2 \implies \frac{1}{\delta(z)}<\frac{d(D)}{\eta(z)}\leq\frac{2}{\delta(z)}.
		\end{align}
		The first half of the inequality follows from \eqref{mdelta} and the increasing property of the logarithm. Indeed, we have
		\begin{align*}
			j_D(x,y)=\log\left(1+\frac{|x-y|}{\delta(x)}\right)\leq\log\left(1+\frac{d\,|x-y|}{\eta(x)}\right)=\zeta_D(x,y).
		\end{align*}
		The equality holds if any one of the points is very close to the boundary.
		
		For the second part of the inequality, we employ the well-known Bernoulli's inequality, the increasing property of the logarithm, and \eqref{mdelta} to get
		\begin{align*}
			j_D(x,y)=\log\left(1+\frac{|x-y|}{\delta(x)}\right)
			&\geq \frac{1}{2}\log\left(1+\frac{2|x-y|}{\delta(x)}\right)\\
			&\geq\frac{1}{2}\log\left(1+\frac{d\,|x-y|}{\eta(x)}\right)=\frac{1}{2}\zeta_D(x,y).
		\end{align*}
       We consider $D=\mathbb{D}$ and choose one point as the origin and another as a positive real number $t$ to prove the sharpness. Then one can easily see that
		$$j_{\mathbb{D}}(0,t)=\log\left(\frac{1}{1-t}\right),\ \text{  and   }\ \zeta_{\mathbb{D}}(0,t)=\log\left(1+\frac{2t}{1-t^2}\right).$$
		Finally, the limit of $j_{\mathbb{D}}(0,t)/\zeta_{\mathbb{D}}(0,t)$ is $1/2$ as $t\rightarrow0$ and the sharpness follows.	
        
	\item[(ii)] For the first part of the inequality, it follows from \eqref{mdelta} that
	\begin{align*}
		\frac{1}{\delta(x)}\leq\frac{d(D)}{\eta(x)}
		\implies 1+\frac{|x-y|}{\delta(x)}\leq 1+\frac{d(D)\,|x-y|}{\eta(x)},
	\end{align*} 	
	which is also true when we replace $x$ by $y$. Applying logarithm to both the inequalities on both sides, adding them, and multiplying by $1/2$ yields the result. The sharpness can be seen from the fact that $j'_{\mathbb{D}}(0,t)/\zeta'_{\mathbb{D}}(0,t)\rightarrow1$ as $t\rightarrow1$.
	
	For the last inequality, we use the second part of \eqref{mdelta} to get
	\begin{align*}
		\zeta'_D(x,y)
		&=\frac{1}{2}\log\left(1+\frac{d(D)\,|x-y|}{\eta(x)}\right)+\frac{1}{2}\log\left(1+\frac{d(D)\,|x-y|}{\eta(y)}\right)\\
		&\leq \frac{1}{2}\log\left(1+\frac{2\,|x-y|}{\delta(x)}\right)+\frac{1}{2}\log\left(1+\frac{2\,|x-y|}{\delta(x)}\right)\\
		&\leq \log\left(1+\frac{|x-y|}{\delta(x)}\right)+\log\left(1+\frac{|x-y|}{\delta(x)}\right)\\
		&=2\,j'_D(x,y),
	\end{align*}
	where the second inequality follows from Bernoulli's inequality. The constant $2$ cannot be improved because of the fact that $\zeta'_{\mathbb{D}}(-t,t)/j'_{\mathbb{D}}(-t,t)\rightarrow2$ as $t\rightarrow 0$.
    \end{enumerate}
    This completes the proof.
	\end{proof}
	
	\begin{corollary}\label{equiv of j' and zeta and v-v}
		For any two points $x,y$ in a bounded domain $D$ in $\mathbb{R}^n$, we have
		\begin{enumerate}
			\item[(i)] $j'_D(x,y)\leq\zeta_D(x,y)\leq 4\,j'_D(x,y)$, and 
		
			\item[(ii)] $\cfrac{1}{2}\,j_D(x,y)\leq \zeta'_D(x,y)\leq 2\,j_D(x,y)$.
		\end{enumerate}
		The first inequality in (i) is sharp, and both the constants in (ii) are the best possible.
	\end{corollary}
	\begin{proof}
		The proof is an immediate consequence of Theorem \ref{equiv of zeta&j} and the inequality (i) of Lemma \ref{j&j' equiv & reln wth k}. The limit of $\zeta_{\mathbb{D}}(0,t)/j'_{\mathbb{D}}(0,t)\rightarrow 1$, as $t\rightarrow 1$, demonstrates the sharpness for the first inequality in (i). The sharpness of the inequalities in (ii) follows from the fact that $\lim_{t\rightarrow 1}\zeta'_{\mathbb{D}}(0,t)/j_{\mathbb{D}}(0,t)=1/2$ and $\lim_{t\rightarrow 0}\zeta'_{\mathbb{D}}(-t,t)/j_{\mathbb{D}}(-t,t)=2$, respectively.
	\end{proof}
	Now we show that the second inequality of Corollary \ref{equiv of j' and zeta and v-v} (i) can be improved with a sharp constant 2.
    \begin{proposition}\label{sharpness in zeta-j' inequality}
        For any two points $x,y$ in a bounded domain $D$, we have
         $$\zeta_D(x,y)\leq 2\,j'_D(x,y).$$
    \end{proposition}
    \begin{proof}
        Without loss of any generality, we may assume that $\delta(x)\leq \delta(y)$. Hence, by Lemma \ref{Prop-of-eta-delta} (i), we have $\eta(x)\leq \eta(y)$. Now we start with the second part of the inequality \eqref{mdelta} and use Lemma \ref{Prop-of-eta-delta} (ii) to obtain
        \begin{align*}
            \frac{d(D)}{d(D)-\delta(x)}\leq 2
            \leq 1+\frac{\delta(x)}{\delta(y)}+\frac{|x-y|}{\delta(y)}
            \implies &\ \frac{d(D)}{\eta(x)\wedge\eta(y)}=\frac{d(D)}{\eta(x)} \leq \frac{1}{\delta(x)}+\frac{1}{\delta(y)}+\frac{|x-y|}{\delta(x)\delta(y)}\\
            \implies & \ 1+\frac{d(D)\,|x-y|}{\eta(x)\wedge\eta(y)}\leq 1+\frac{|x-y|}{\delta(x)}+\frac{|x-y|}{\delta(y)}+\frac{|x-y|^2}{\delta(x)\delta(y)}\\
            \implies & \ 1+\frac{d|x-y|}{\eta(x)\wedge\eta(y)}\leq \left(1+\frac{|x-y|}{\delta(x)}\right)\left(1+\frac{|x-y|}{\delta(y)}\right).
        \end{align*}
        Finally, taking logarithm on both sides gives the result. The sharpness is evident because of the limit $\zeta'_{\mathbb{D}}(-t,t)/j_{\mathbb{D}}(-t,t)\rightarrow 2$ as $t\rightarrow 0$.
    \end{proof}

\subsection{Comparison with the $k_D$-metric}	
	We have the following inequality between $\zeta_D$ and $k_D$ in bounded domains. 
  \begin{proposition}\label{equiv of k with zeta 2}
	For any bounded domain $D$ in $\mathbb{R}^n$ and for $x,y$ in $D$, we have 
	\begin{align*}
		\zeta_D(x,y)\leq 2\,k_D(x,y).
	\end{align*}
	The inequality is sharp.
  \end{proposition}
  \begin{proof}
  	Direct use of Theorem \ref{equiv of zeta&j} and Lemma \ref{j&j' equiv & reln wth k} (ii) yields the inequality. To see the sharpness, we use the fact that on a radial line in $\mathbb{D}$, $k_{\mathbb{D}}$ matches with $j_{\mathbb{D}}$. Therefore, we have
  	\begin{align*}
  		\lim_{t\rightarrow0}\frac{\zeta_{\mathbb{D}}(0,t)}{k_{\mathbb{D}}(0,t)}=\lim_{t\rightarrow0}\frac{\zeta_{\mathbb{D}}(0,t)}{j_{\mathbb{D}}(0,t)}=2,
  	\end{align*}	
  	as we have seen in Theorem \ref{equiv of zeta&j} (i).
  \end{proof}
  In Proposition \ref{equiv of k with zeta 1}, we will see that the constant $2$ can be replaced with $1$ in the case of non-uniform bounded domains.    


\section{\bf Inner metric of $\zeta_D$}\label{inner metric-sec}		
   In this section, our primary goal is to show that in any bounded domain $D$ of $\mathbb{R}^n$, $n\geq 2$, the metric $m_D$ is the inner metric of $\zeta_D$. To prove the main result, we begin with the following lemma:
	\begin{lemma}\label{lemma-for-inner metric proof}
		Let $D\subset\mathbb{R}^n$, $n\geq 2$, be a bounded domain and $x$ be any point in $D$.
		\begin{enumerate}
			\item[(i)] Then for any $y\in B_x=B(x,\delta(x))$, we have
			\begin{align*}
				m_D(x,y)\leq\log\left(1+\frac{d(D)\,|x-y|}{\eta(x)-d(D)|x-y|}\right).
			\end{align*}
			\item[(ii)] For an arbitrarily small $s\in(0,1)$ and for any point $y\in B(x,s\,\eta(x)/d(D))$, we have
			\begin{align*}
				m_D(x,y)\leq \frac{1}{1-s}\,\zeta_D(x,y).
			\end{align*}
		\end{enumerate}
	\end{lemma}
	\begin{proof}
    \begin{enumerate}
		\item [(i)]
			Let us choose a point $w$ on the circle $\partial B(x,|x-y|)$ such that the point $w$ lies on the line segment joining $x$ to the nearest boundary point, say $\xi$. Then for any point $z\in[x,w]$, we have $\delta(z)\leq\delta(x)$ and hence $\eta(z)\leq\eta(x)$, by Lemma \ref{Prop-of-eta-delta} (i). Again, using (ii) of the same lemma, we have $\eta(x)-\eta(z)\leq d(D)|x-z|$. With all these in our hands, we compute
			\begin{align*}
				m_D(x,y)\leq d(D)\int_{[x,y]}\frac{ds}{\eta(z)}
				= d(D)\int_{[x,w]}\frac{ds}{\eta(z)}
				\leq &d(D)\int_{[x,w]}\frac{ds}{\eta(x)-d(D)\,|x-z|}\\
				=&\int_{\eta(x)-d(D)\,|x-w|}^{\eta(x)}\frac{dt}{t}\\
				=&\log\left(\frac{\eta(x)}{\eta(x)-d(D)\,|x-y|}\right)\\
				=&\log\left(1+\frac{d(D)\,|x-y|}{\eta(x)-d(D)|x-y|}\right).
			\end{align*}
			The second step follows from the fact that rotation is an isometry of the $m_D$-metric, and in the second last step we used $|x-y|=|x-w|$, since $w\in\partial B(x,|x-y|)$.
			
			\item[(ii)] Since $s\,\eta(x)/d(D)<\eta(x)/d(D)\leq \delta(x)$, we see that $y\in B(x,s\,\eta(x)/d(D))$ implies $y\in B(x,\delta(x))$. Applying (i), we obtain
			\begin{align*}
				m_D(x,y)\leq \log\left(1+\frac{d(D)\,|x-y|}{\eta(x)-d(D)|x-y|}\right)
				\leq& \log\left(1+\frac{d(D)\,|x-y|}{(1-s)\eta(x)}\right)\\
				\leq& \frac{1}{1-s}\log\left(1+\frac{d(D)\,|x-y|}{\eta(x)}\right)\\
				\leq& \frac{1}{1-s}\zeta_D(x,y),
			\end{align*}
			where the third inequality follows from Bernoulli's inequality.
        \end{enumerate}
        This completes the proof.
	\end{proof}
	
	\begin{corollary}
		Let $D\subsetneq\mathbb{R}^n$, $n\geq 2$, be a domain and $x$ be any point in $D$.
		\begin{enumerate}
			\item[(i)] Then for any $y\in B_x=B(x,\delta(x))$, we have
			\begin{align*}
				k_D(x,y)\leq\log\left(1+\frac{|x-y|}{\delta(x)-|x-y|}\right).
			\end{align*}
			\item[(ii)] For an arbitrarily small $s\in(0,1)$ and for any point $y\in B(x,s\,\delta(x))$, we have
			\begin{align*}
				k_D(x,y)\leq \frac{1}{1-s}\,j_D(x,y).
			\end{align*}
		\end{enumerate}
	\end{corollary} 
	
	Now, we make use of Lemma \ref{lemma-for-inner metric proof} to prove that the metric $m_D$ is the inner metric of $\zeta_D$. 
	\begin{theorem}\label{inner metric}
		Let $D$ be a proper subdomain of $\mathbb{R}^n$, $n\geq 2$. Then for any two points $x,y\in D$, we have $\tilde{\zeta}_D(x,y)=m_D(x,y)$.
	\end{theorem}
	\begin{proof}
		First we prove that $\tilde{\zeta}_D(x,y)\leq m_D(x,y)$ holds for any points $x,y\in D$. Let $\gamma\in\Gamma_{xy}$ be an $m_D$-geodesic in $D$ and choose successive points $x=x_0,x_1,\ldots,x_n=y$, with $x_i=\gamma(t_i)$. Then using Theorem \ref{equiv of zeta&m}, we can write
		\begin{align*}
			\sum_{i=1}^{n}\zeta_D(x_{i-1},x_i)
			\leq\sum_{i=1}^{n}m_D(x_{i-1},x_i)
			=\sum_{i=1}^{n}\ell_m\left(\gamma[x_{i-1},x_i]\right).
		\end{align*}
		Taking supremum over all partitions of the path $\gamma$, we get
		\begin{align*}
			\ell_{\zeta}(\gamma)\leq \ell_m(\gamma)=m_D(x,y).
		\end{align*}
		Finally, considering infimum over $\Gamma_{xy}$, one can obtain 
		\begin{align*}
			\tilde{\zeta}_D(x,y)=\inf_{\beta\in \Gamma_{xy}}\ell_{\zeta}(\beta)\leq\ell_{\zeta}(\gamma)\leq m_D(x,y).
		\end{align*}
		
		Conversely, to see the case $\tilde{\zeta}_D(x,y)\geq m_D(x,y)$, let $\epsilon>0$ be very small. We choose a finite sequence $x=x_0,\,x_1,\ldots,x_n=y$ of points on any path $\beta\in\Gamma_{xy}\subset D$ in such a way that $|x_{s-1}-x_s|\leq \epsilon\,\eta(x_{s-1})/d(D)$ with $\beta(t_i)=x_i$. Then, using (ii) of Theorem \ref{lemma-for-inner metric proof}, we get
		\begin{align*}
			\ell_{\zeta}(\beta)
			=\sup_{\mathcal{P}}\sum_{i=1}^{n}\zeta_D(x_{i-1},x_i)
			\geq& (1-\epsilon)\,\sup_{\mathcal{P}}\sum_{i=1}^{n}m_D(x_{i-1},x_i)\\
			\geq& (1-\epsilon)\,m_D(x,y).
		\end{align*}
	 Taking infimum over all $\beta\in\Gamma_{xy}$, we get
		\begin{align*}
			\tilde{\zeta}_D(x,y)\geq (1-\epsilon)\,m_D(x,y).
		\end{align*}
		Letting $\epsilon\rightarrow 0$ yields the required inequality.
	\end{proof} 
    \begin{remark}
        Theorem \ref{equiv of zeta&m} also follows using Theorem \ref{inner metric} and the fact that $d\leq\tilde{d}$.
    \end{remark}
	Now we see that the density of $\zeta_D$ is the same as the density of $m_D$ in infinitesimal form. Indeed, we have the following proposition.
	\begin{proposition}\label{two sided log inqlty of zeta}
		For any two points $x,y\in D$ with $y\in B_x=B(x,\delta(x))$, we have the following inequality:
		\begin{align*}
			\log\left(1+\frac{d(D)\,|x-y|}{\eta(x)+d(D)|x-y|}\right)\leq \zeta_D(x,y)\leq\log\left(1+\frac{d(D)\,|x-y|}{\eta(x)-d(D)|x-y|}\right).
		\end{align*}  
		In particular, 
		\begin{align*}
			\lim_{y\rightarrow x}\frac{\zeta_D(x,y)}{|x-y|}=\frac{d(D)}{\eta(x)}.
		\end{align*}
	\end{proposition}
	\begin{proof}
		From the definition of the $\zeta_D$-metric, we have the trivial first inequality
		\begin{align*}
			\log\left(1+\frac{d(D)\,|x-y|}{\eta(x)+d(D)|x-y|}\right)\leq
			\log\left(1+\frac{d(D)\,|x-y|}{\eta(x)}\right) 
			\leq\zeta_D(x,y).
		\end{align*}
		The second inequality follows from Theorem \ref{equiv of zeta&m} and Lemma \ref{lemma-for-inner metric proof}.
		
One can replace $|x-y|$ with another variable $t\rightarrow 0$ and employ the squeeze rule to prove the limit.
	\end{proof}
	
	\begin{corollary}\label{two sided log inqlty of m_D}
		For any two points $x,y\in D$ with $y\in B_x=B(x,\delta(x))$, we have the following inequality for $m_D$-metric:
		\begin{align*}
			\log\left(1+\frac{d(D)\,|x-y|}{\eta(x)+d(D)|x-y|}\right)\leq m_D(x,y)\leq\log\left(1+\frac{d(D)\,|x-y|}{\eta(x)-d(D)|x-y|}\right).
		\end{align*}
	\end{corollary}
	\begin{proof}
		The proof is straightforward by employing Theorem \ref{equiv of zeta&m}, Lemma \ref{lemma-for-inner metric proof}, and Proposition \ref{two sided log inqlty of zeta}.
	\end{proof}

	Taking $d(D)\rightarrow\infty$ in Proposition \ref{two sided log inqlty of zeta} and in Corollary \ref{two sided log inqlty of m_D}, we obtain an analogous inequality for the metrics $j_D$ and $k_D$. 
	\begin{corollary}
		Let $x,y\in D\subsetneq\mathbb{R}^n$. Then, for $d\in\left\{j_D,k_D\right\}$, we have
		\begin{align*}
			\log\left(1+\frac{|x-y|}{\delta(x)+|x-y|}\right)\leq d(x,y)\leq\log\left(1+\frac{|x-y|}{\delta(x)-|x-y|}\right).
		\end{align*}    
	\end{corollary}
    

\section{\bf Ball inclusion property}\label{ball inclusion}	
Now we shall observe how the metric balls, related to our metrics $\zeta_D$, $\zeta'_D$, and $m_D$, behave with each other and with other hyperbolic-type metrics. For a metric $d$, defined on a domain $D$, we denote a $d$-metric ball of radius $r$ and center at $z_0$ by $B_d(z_0,r):=\left\{z\in D:d(z,z_0)<r\right\}$.
\begin{theorem}\label{ball inclusion_zeta-euclidean}
	Let $D\subsetneq\mathbb{R}^n$ be a domain and $x\in D$. For $s>0$, we consider a $\zeta_D$-metric ball $B_{\zeta}(x,s)$. Then the following inclusions hold:
	\begin{align}
		B(x,r)\subset B_{\zeta}(x,s)\subset B(x,R),
	\end{align}
	with $r=(1-e^{-s})(\eta(x)/d(D))$ and $R=(e^s-1)(\eta(x)/d(D))$. The values of $r$ and $R$ are the best possible. Moreover, $R/r\rightarrow1$ as $s\rightarrow 0$.
\end{theorem}
\begin{proof}
	For the first inclusion, let $y\in B(x,r)$.  Then, we see that
	\begin{align*}
		|x-y|<r=(1-e^{-s})\frac{\eta(x)}{d(D)}
		&\implies e^{-s}<\frac{\eta(x)-d(D)\,|x-y|}{\eta(x)}\\
		&\implies s>\log\left(1+\frac{d(D)\,|x-y|}{\eta(x)-d(D)\,|x-y|}\right),
	\end{align*}
	which implies $\zeta_D(x,y)<s$ by using Proposition \ref{two sided log inqlty of zeta}. Thus, the first inclusion follows with the specified $r$.
	
	To see the next inclusion, we take a point $y\in B_{\zeta}(x,s)$, and for convenience, we let $\zeta=\zeta_D(x,y)<s$. We need to show that $|x-y|<R$, where $R$ is as mentioned in the statement. From the definition of $\zeta$, we have 

		$$
        1+\frac{d(D)\,|x-y|}{\eta(x)}\leq 1+\frac{d(D)\,|x-y|}{\eta(x)\wedge\eta(y)}=e^{\zeta}
		\implies|x-y|\leq (e^{\zeta}-1)\frac{\eta(x)}{d(D)}<R.
        $$

	Hence, the second inclusion follows. The fact that given $r$ and $R$ are the best possible is evident from \cite[Theorem 3.8]{Se} when $D=\mathbb{R}^n\setminus \{0\}$.
\end{proof}

In the next theorem, we show that the above ball inclusion property of the $\zeta_D$-metric is also true for the $m_D$-metric balls with the same values of $r$ and $R$.
\begin{theorem}\label{ball inclusion_m_D-euclidean}
    Let $D\subsetneq\mathbb{R}^n$ be a domain and $x\in D$. For $s>0$, we consider the $m_D$-metric ball $B_m(x,s)$. Then the following inclusions hold:
	\begin{align}
		B(x,r)\subset B_{m}(x,s)\subset B(x,R),
	\end{align}
	with $r=(1-e^{-s})(\eta(x)/d(D))$ and $R=(e^s-1)(\eta(x)/d(D))$. The values of $r$ and $R$ are the best possible. Moreover, $R/r\rightarrow1$ as $s\rightarrow 0$.
\end{theorem}
\begin{proof}
    Let $y$ be any point in $D\cap B(x,r)$. Then, from Lemma \ref{lemma-for-inner metric proof} (i), we have
    \begin{align*}
        |x-y|<r\implies m_D(x,y)<s,
    \end{align*}
    which gives the first inclusion property.
    The next inclusion follows from Theorem \ref{equiv of zeta&m} and Theorem \ref{ball inclusion_zeta-euclidean}. One can refer to \cite[(4.11)]{HaKlVu} to see that the given values of $r$ and $R$ are the best possible radii.
\end{proof}

\begin{theorem}\label{ball inclusion_m_D-zeta_D}
    Let $D\subset \mathbb{R}^n$ and $s\in(0,\log 2)$. Then we have the following inclusion property:
    \begin{align*}
        B_{\zeta}(x,r)\subset B_{m}(x,s)\subset B_{\zeta}(x,s)\subset B_{m}(x,R),
    \end{align*}
    with 
    $$r=\log(2-e^{-s})\textrm{ and } R=\log\frac{1}{2-e^{s}}.$$
    Moreover, $R/r\rightarrow1$ as $s\rightarrow 0$.
\end{theorem}
\begin{proof}
    First, note that the choice of $s$ makes both $r$ and $R$ positive. We make use of Theorems \ref{ball inclusion_zeta-euclidean} and \ref{ball inclusion_m_D-euclidean} step-by-step to conclude
    \begin{align}\label{ball_m-zeta1}
        B_{\zeta}(x,r)\subset B(x,(e^r-1)(\eta(x)/d(D)))=B(x,(1-e^{-s})(\eta(x)/d(D)))\subset B_{m}(x,s),
    \end{align}
    and 
    \begin{align}\label{ball_m-zeta2}
        B_{m}(x,s)\subset B_{\zeta}(x,s)\subset B_{\zeta}(x,(e^s-1)(\eta(x)/d(D)))=B(x,(1-e^{-R})(\eta(x)/d(D)))\subset B_m(x,R),
    \end{align}
    where the first inclusion in \eqref{ball_m-zeta2} follows by Theorem \ref{equiv of zeta&m}. Note that all the equalities of balls in the above are due to the chosen values of $r$ and $R$. Thus \eqref{ball_m-zeta1} and \eqref{ball_m-zeta2} together give all the required inclusion relations. To see the limit,
    \begin{align*}
        \lim_{s\rightarrow 0}\frac{R}{r}
        =-\lim_{s\rightarrow 0}\frac{\log(2-e^s)}{\log(2-e^{-s})}
        =\lim_{s\rightarrow 0}\cfrac{\cfrac{e^s}{2-e^s}}{\cfrac{e^{-s}}{2-e^{-s}}}=1.
    \end{align*}
    This concludes the result.
\end{proof}

\begin{proposition}\label{ball inclusion_zeta-zeta'}
    Let $D\subsetneq\mathbb{R}^n$ be a domain and $x\in D$. Then the following inclusions hold for $s>0$:
         $$B(x,r)\subset B_{\zeta'}(x,s)\subset B(x,R),$$
    with $r=(1-e^{-s})(\eta(x)/d(D))$ and $R=(e^{2s}-1)(\eta(x)/d(D))$.
\end{proposition}
\begin{proof}
    We use Proposition \ref{eqvlnc of two metric} to conclude the result. Therefore, we have
    $$B{\zeta}(x,s)\subset B_{\zeta'}(x,s)\subset B{\zeta}(x,2s).$$
    Now applying Theorem \ref{ball inclusion_zeta-euclidean}, one can obtain the required inclusions.
\end{proof}

\begin{proposition}\label{ball inclusion_zeta-j,j'}
    Let $D\subsetneq\mathbb{R}^n$ be a domain and $x\in D$. Then the following inclusions hold:
    \begin{enumerate}
        \item[(i)]  $B_j(x,r)\subset B_{\zeta}(x,s)\subset B_j(x,R)$, and
        \item[(ii)]  $B_{j'}(x,r)\subset B_{\zeta}(x,s)\subset B_{j'}(x,R)$,
    \end{enumerate}
    with $r=s/2$ and $R=s$.
\end{proposition}
\begin{proof}
    The proof follows from Theorem \ref{equiv of zeta&j} (i), Corollary \ref{equiv of j' and zeta and v-v} (i), and Proposition \ref{sharpness in zeta-j' inequality}.
\end{proof}

\begin{proposition}\label{ball inclusion_zeta'-j,j'}
    Let $D\subsetneq\mathbb{R}^n$ be a domain and $x\in D$. Then the following inclusions hold:
    \begin{enumerate}
        \item[(i)]  $B_j(x,r)\subset B_{\zeta'}(x,s)\subset B_j(x,R)$, and
        \item[(ii)]  $B_{j'}(x,r)\subset B_{\zeta'}(x,s)\subset B_{j'}(x,R)$,
    \end{enumerate}
    with $r=s/2$ and $R=2s$.
\end{proposition}
\begin{proof}
    Using Theorem \ref{equiv of zeta&j} (ii), and Corollary \ref{equiv of j' and zeta and v-v} (ii), one can conclude the inclusions.
\end{proof}


\section{\bf Behavior under special maps}\label{behvr undr spcl maps}
   \subsection{Behavior under M\"{o}bius transformations} 
   In 1979, Gehring and Osgood studied the behavior of $j_D'$ under quasiconformal maps \cite[Theorem 3]{GO}. Upon carefully observing their proof, it can be seen that 
   $$j'_{D'}(f(x),f(y))\leq 2\,j'_D(x,y),$$
   under a M\"{o}bius transformation $f$ of $\overline{\mathbb{R}^n}$ and for any two points $x,y\in D$, where $D'=f(D)$. One can modify their proof and see that the metric $j_D$ behaves alike under M\"{o}bius transformations. We expect similar properties for our metrics $\zeta_D$ and $\zeta_D'$, which are presented in the following form:
 \begin{theorem}\label{zeta-under-Mob}
 	Let $f$ be a M\"{o}bius transformation of $\overline{\mathbb{R}}^n$ and $D$ be a proper bounded subdomain with $f(D)=D'$. Then for any two points $x,y\in D$, we have
    \begin{enumerate}
 	\item[(i)] $\zeta_{D'}(f(x),f(y))\leq 4\,\zeta_D(x,y)$, and\label{mob equiv}
        \item[(ii)] $\zeta'_{D'}(f(x),f(y))\leq 4\,\zeta'_D(x,y)$.
     \end{enumerate}
 \end{theorem}
 \begin{proof}
 \begin{enumerate}
 	\item[(i)] Let $x,y$ be two fixed points in $D$. First we ensure that $\eta(f(x))\leq \eta(f(y))$, if not, first relabel them. Let us choose $\xi\in\partial D$ and $\tau \in \overline{\mathbb{R}}^n\setminus D$ such that 
 	\begin{equation*}
 		|f(x)-f(\xi)|=\delta_{D'}(f(x)),
 	\end{equation*}
 	and $f(\tau)=\infty$. Since M\"{o}bius transformations preserve cross-ratio, one can obtain
 	\begin{align}\label{cross-ratio}
 		\frac{|x-y||\xi-\tau|}{|x-\xi||y-\tau|}=\frac{|f(x)-f(y)|}{|f(x)-f(\xi)|}
 		\geq\frac{d(D')}{2}\frac{|f(x)-f(y)|}{\eta_{D'}(f(x))},
 	\end{align}
 	where the last inequality follows from \eqref{mdelta}. We now consider two cases on the point $\tau$.
 	
 	{\bf Case-I:} First, let us assume $\tau=\infty$. Then from equation \eqref{cross-ratio}, we have
 	\begin{align*}
 		\frac{|x-y|}{|x-\xi|}=\frac{d(D')}{2}\frac{|f(x)-f(y)|}{\eta(f(x))}.
 	\end{align*}
 	From Lemma \ref{Prop-of-eta-delta} (iv), we obtain
 	\begin{align}\label{mob I(i)}
 	\frac{d(D')}{2}\frac{|f(x)-f(y)|}{\eta_{D'}(f(x))}&\leq\frac{d(D)|x-y|}{\eta(x)}\\ 		
 	\implies 1+\frac{d(D')|f(x)-f(y)|}{2\,\eta_{D'}(f(x))}&\leq1+\frac{d(D)|x-y|}{\eta_D(x)\wedge\eta_D(y)}.\nonumber 
 	\end{align}
Taking logarithm on both sides and using Bernoulli's inequality to get
    \begin{align*}
    	\zeta_{D'}(f(x),f(y))\leq 2\,\zeta_D(x,y).
    \end{align*}

  {\bf Case-II:} Now, let us assume $\tau\neq\infty$. By the Euclidean triangle inequality, \eqref{cross-ratio} leads to
  \begin{align}\label{mob II(i)}
  	1+\frac{d(D')|f(x)-f(y)|}{2\,\eta_{D'}(f(x))}
  	&\leq1+\frac{|x-y|}{|y-\tau|}+\frac{|x-y|^2}{|x-\xi||y-\tau|}+\frac{|x-y|}{|x-\xi|}\nonumber\\
  	&=\left(1+\frac{|x-y|}{|x-\xi|}\right)\left(1+\frac{|x-y|}{|y-\tau|}\right)\nonumber\\
  	&\leq\left(1+\frac{d(D)|x-y|}{\eta_D(x)}\right)\left(1+\frac{d(D)|x-y|}{\eta_D(y)}\right)\\
  	&\leq\left(1+\frac{d(D)|x-y|}{\eta_D(x)\wedge\eta_D(y)}\right)^2. \nonumber 
  \end{align}
  To obtain our required inequality, we take logarithm on both sides of the preceding inequality and use Bernoulli's inequality.
  
  \item[(ii)] The proof follows from a similar technique as in part (i). 
  
\end{enumerate}
\end{proof}
 
\subsection{Behavior under quasiconformal maps}
   In \cite{GO}, Gehring and Osgood studied the quasi-invariance properties of $j_D$ and $k_D$ under quasiconformal maps. We observed in \cite[Theorem 3.13]{MaNaSa} that the behavior of $m_D$ is similar to the quasihyperbolic metric. Indeed, it has been proved that
	\begin{theorem}
		If $f$ is a $K$-quasiconformal mapping of $D_1$ to $D_2$, both are proper subdomains of $\mathbb{R}^n$, then there exists a constant $c$, depending only on $n$ and $K$, with the following property:
		\begin{align*}
			m_{D_2}(f(x),f(y))\leq c\,\max\left\{m_{D_1}(x,y),m_{D_1}(x,y)^{\alpha}\right\}, \hspace{0.5 cm}\alpha=K^{1/(1-n)},	
		\end{align*}
		for all $x,y\in D$.
	\end{theorem}	
A similar result for the metrics $\zeta_D$ and $\zeta_D'$ follows from \cite[Theorem 4]{GO}, Theorem \ref{equiv of zeta&j}, and Corollary \ref{equiv of j' and zeta and v-v}.
\begin{theorem}\label{dist of zeta under qc map}
	Let $f$ be a $K$-quasiconformal maps of $\mathbb{R}^n$ with $f(D)=D'$, where $D$ and $D'$ are domains in $\mathbb{R}^n$ with $d(D)<\infty$. Then there exist constants $a$ and $b$ such that for any two points $x,y\in D$ the following holds:
	\begin{align*}
		\beta_{D'}(f(x),f(y))\leq a\,\beta_D(x,y)+b,
	\end{align*} 
	where $a=2c$ and $b=2d$ and the constants $c,d$ are as in \cite[Theorem 4]{GO}. Here, $\beta_D\in\{\zeta_D,\zeta_D'\}$.
\end{theorem}

\subsection{Behavior under quasiregular maps}	
This section shows the distortion of the metrics $\zeta_D$ and $\zeta_D'$ under quasiregular maps on the unit disk. Indeed, our result is a Schwarz-type lemma for the new metrics under quasiregular mappings of the unit disk. We directly use the following result due to Bhayo and Vuorinen.
\begin{theorem}\cite[Theorem 1.10]{BV}
	If $f:\mathbb{D}\rightarrow\mathbb{R}^2$ is a non-constant $K$-quasiregular mapping with $f(\mathbb{D})\subseteq\mathbb{D}$, then 
	\begin{align*}
		h_{\mathbb{D}}(f(x),f(y))\leq c(K)\,\max\left\{h_{\mathbb{D}}(x,y),h_{\mathbb{D}}(x,y)^{1/K}\right\}
	\end{align*}
	for all $x,y\in \mathbb{D}$, with an explicit value of $c(K)$ satisfying $c(1)=1$.
\end{theorem}
First, we recall that the unit disk $\mathbb{D}$ is a well-known example of a uniform domain. Therefore, by Theorem \ref{uniform charc}, there exists $c'\geq 1$ such that $h_{\mathbb{D}}(x,y)=m_{\mathbb{D}}(x,y)\leq c'\,\zeta_{\mathbb{D}}(x,y)$, for all $x,y\in\mathbb{D}.$ Therefore, the use of Theorem \ref{equiv of zeta&m} and the above theorem provides 
\begin{align*}
	\zeta_{\mathbb{D}}(f(x),f(y))
	\leq m_{\mathbb{D}}(f(x),f(y))
	=h_{\mathbb{D}}(f(x),f(y))\
	&\leq c(K)\,\max\left\{h_{\mathbb{D}}(x,y),h_{\mathbb{D}}(x,y)^{1/K}\right\}\\
	&\leq c'\,c(K)\,\max\left\{\zeta_{\mathbb{D}}(x,y),\zeta_{\mathbb{D}}(x,y)^{1/K}\right\}.
\end{align*}
Indeed, we just obtained the following distortion result for the $\zeta_D$-metric under a $K$-quasiregular map.
\begin{theorem}
		If $f:\mathbb{D}\rightarrow\mathbb{R}^2$ is a non-constant $K$-quasiregular mapping with $f(\mathbb{D})\subseteq\mathbb{D}$, then we have
	\begin{align*}
		\beta_{\mathbb{D}}(f(x),f(y))\leq c'(K)\,\max\left\{\beta_{\mathbb{D}}(x,y),\beta_{\mathbb{D}}(x,y)^{1/K}\right\},
	\end{align*}
	for all $x,y\in \mathbb{D}$ and $c'(K)$ is a constant depending only on $K$. Here, $\beta_{\mathbb{D}}\in\{\zeta_{\mathbb{D}},\zeta_{\mathbb{D}}'\}$.
\end{theorem}


\section{\bf Applications related to uniform domains}\label{charc-of-domains}	
\subsection{Characterization of uniform domains}
Recall that there are many characterizations of uniform domains. One of the most important characterizations is due to Gehring and Osgood \cite{GO}, which is in terms of metric inequality related to the quasihyperbolic metric and the distance ratio metric. We provide a characterization of uniform domains through inequality, concerning our new metrics $m_D$, $\zeta_D$, and $\zeta_D'$, which matches with their result in unbounded domains.
\begin{theorem}\label{uniform charc}
	A domain $D\subsetneq\mathbb{R}^n$, $n\geq 2$, is uniform if and only if for any pair points $x,y$ in $D$, there exists a universal constant $c\geq 1$ such that
	\begin{align*}
		m_D(x,y)\leq c\,\beta_D(x,y).
	\end{align*}
    Here, $\beta_D\in\{\zeta_D,\zeta_D'\}$.
\end{theorem}
\begin{proof}
If $D$ is bounded, the proof follows directly from \cite[Theorem 3.5]{MaNaSa} and Theorem \ref{equiv of zeta&j}. For an unbounded domain $D$, the result agrees with \eqref{GO-uniform}.
\end{proof}
Now, we apply Theorem \ref{uniform charc} and improve the result in Proposition \ref{equiv of k with zeta 2} as follows.
  \begin{proposition}\label{equiv of k with zeta 1}
	For any bounded non-uniform domain $D$ in $\mathbb{R}^n$ and for $x,y$ in $D$, we have 
	\begin{align*}
		\zeta_D(x,y)\leq k_D(x,y).
	\end{align*}
\end{proposition}
\begin{proof}
	Suppose, on the contrary, the reverse inequality holds in $D$, i.e., for any two points $x,y$ in $D$, we have $k_D(x,y)< \zeta_D(x,y)$. Then 
	\begin{align*}
		\frac{1}{2}\,m_D(x,y)\leq k_D(x,y)< \zeta_D(x,y),
	\end{align*}
	where the first inequality follows from \cite[Theorem 3.5]{MaNaSa}. Hence, from Theorem \ref{uniform charc}, $D$ is a uniform domain, a contradiction.
\end{proof}

\begin{remark}
	Taking $d(D)\rightarrow\infty$, we again obtain Corollary \ref{j-leq-k}.
\end{remark}

\section{\bf Concluding remarks}
{\bf Summary:} This manuscript introduces a new variant of Vuorinen’s distance ratio metric $j_D$, denoted by $\zeta_D$, along with a corresponding variant of the Gehring–Osgood metric $j_D'$, denoted by $\zeta_D'$. A primary motivation for considering these metrics is to investigate the interaction between the quasihyperbolic metric and bounded uniform domains in $\mathbb{R}^n$ for $n \ge 2$. It is further shown that the inner metric associated with $\zeta_D$ coincides with the recently introduced metric $m_D$.

We established several equivalence relations between these metrics and other hyperbolic-type metrics, together with inclusion properties of balls defined by $\zeta_D$, $\zeta_D'$, and $m_D$. Distortion properties under important classes of mappings are also investigated. As an application, we demonstrated that uniform domains can be characterized in terms of $\zeta_D$ and $m_D$.

{\bf Future Scope:} It is well known that the inner metric of $j_D$ is the quasihyperbolic metric $k_D$, and we have shown that the inner metric of $\zeta_D$ is $m_D$. However, the inner metrics of $j_D'$ and $\zeta_D'$ remain unknown. We have further examined the distortion properties of the metrics $\zeta_D$ and $\zeta_D'$ under Möbius transformations, though the question of sharpness is still unresolved. Another open problem emerging from our work is to determine the sharpness of ball inclusions related to these metrics. Moreover, it would be of significant interest to characterize the isometries and geodesics associated with these new metrics. Overall, investigations of these problems will certainly open new and fruitful avenues for early-career researchers in this field.

\section*{\bf Acknowledgements}
The work of the second author is supported by the University Grants Commission (UGC) with Ref. No.: 221610091493. 

\section*{\bf Declaration}
\noindent
{\bf Conflict of interest:} The authors declare that there is no conflict of interest regarding the publication of this article.


\begin{thebibliography}{99}
\bibitem{AhlBook} 
L.~V. Ahlfors, {\it Lectures on quasiconformal mappings}, second edition, 
University Lecture Series, 38, Amer. Math. Soc., Providence, RI, 2006.

\bibitem{BM} {A.~F. Beardon and D. Minda, The hyperbolic metric and geometric function theory, in {\it Quasiconformal mappings and their applications}, 9--56, Narosa, New Delhi, 2007.}
	
	\bibitem{BV} {B.~A. Bhayo and M.~K. Vuorinen, On Mori's theorem for quasiconformal maps in the $n$-space, Trans. Amer. Math. Soc. {\bf 363} (2011), no.~11, 5703--5719.}

	\bibitem{GO}{F.~W. Gehring and B.~G. Osgood, Uniform domains and the quasihyperbolic metric, J. Analyse Math. {\bf 36} (1979), 50--74.} 

	\bibitem{GP} {F.~W. Gehring and B.~P. Palka, Quasiconformally homogeneous domains, J. Analyse Math. {\bf 30} (1976), 172--199.}
	
	\bibitem{HaKlVu} {P. Hariri, R. Kl\'en and M.~K. Vuorinen, {\it Conformally invariant metrics and quasiconformal mappings}, Springer Monographs in Mathematics, Springer, Cham, 2020.}

    \bibitem{Has03} P.~A. H\"ast\"o, The Apollonian metric: uniformity and quasiconvexity, Ann. Acad. Sci. Fenn. Math. {\bf 28} (2003), no.~2, 385--414.
	
	\bibitem{HK} {J. Heinonen and P. Koskela, Definitions of quasiconformality, Invent. Math. {\bf 120} (1995), no.~1, 61--79.}
	
	\bibitem{MaNaSa} {B. Maji, P. Naskar, and S. K. Sahoo. A naive generalization of the hyperbolic and the quasihyperbolic metrics.  {\tt arXiv:2505.10964v2}}.

\bibitem{MS79} 
O. Martio and J. Sarvas, Injectivity theorems in plane and space, Ann. Acad. Sci. Fenn. Ser. A I Math. {\bf 4} (1979), no.~2, 383--401.

  \bibitem{Moc24}
   {M. Mocanu, A generalization of Vuorinen's distance ratio metric in metric spaces and bi-Lipschitz equivalent hyperbolic-type metrics, Sci. Stud. Res. Ser. Math. Inform. {\bf 34} (2024), no.~1, 33--46.}
    
	\bibitem{Moc25} {M. Mocanu, Generalizations of four hyperbolic-type metrics and Gromov hyperbolicity, J. Math. Anal. Appl. {\bf 552} (2025), no.~1, Paper No. 129729, 16 pp.}
	

	\bibitem{Se} {P. Seittenranta, M\"obius-invariant metrics, Math. Proc. Cambridge Philos. Soc. {\bf 125} (1999), no.~3, 511--533.}

   \bibitem{VaiBook} 
   J. V\"ais\"al\"a, {\it Lectures on $n$-dimensional quasiconformal mappings}, Lecture Notes in Mathematics, Vol. 229, Springer, Berlin-New York, 1971.

    \bibitem{Vai96}    J. V\"ais\"al\"a, The free quasiworld. Freely quasiconformal and related maps in Banach spaces, in {\it Quasiconformal geometry and dynamics (Lublin, 1996)}, 55--118, Banach Center Publ., 48, Polish Acad. Sci. Inst. Math., Warsaw.
    
	
	\bibitem{Vu85} {M.~K. Vuorinen, Conformal invariants and quasiregular mappings, J. Analyse Math. {\bf 45} (1985), 69--115.}

    \bibitem{VuoBook}
    M.~K. Vuorinen, {\it Conformal geometry and quasiregular mappings}, Lecture Notes in Mathematics, 1319, Springer, Berlin, 1988.
	
\end{thebibliography}
\end{document}